\documentclass[11pt]{article}

\usepackage{amssymb,amsbsy,amsmath,amscd,amsthm,amsfonts}
\usepackage{cite}
\usepackage[utf8]{inputenc}
\usepackage{enumerate,hyperref}

\newtheorem{theorem}{Theorem}

\newtheorem{lemma}{Lemma}

\newtheorem{corollary}{Corollary}

\newcommand{\R}{\mathbb R}
\newcommand{\PP}{\mathbb P}
\newcommand{\E}{\mathbb E}
\newcommand{\K}{\mathcal K_d}

\title{\textbf{A universal deviation inequality for random polytopes}}
\author{Victor-Emmanuel Brunel\\
   CREST, Paris (France) - University of Haifa, Haifa (Israel)\\
   \texttt{victor.emmanuel.brunel@ensae-paristech.fr}}

\date{}

\begin{document}

\maketitle

\begin{abstract}
    We consider the convex hull of a finite sample of i.i.d. points uniformly distributed in a convex body in $\R^d$, $d\geq 2$. We prove an exponential deviation inequality, which leads to rate optimal upper bounds on all the moments of the missing volume of the convex hull, uniformly over all convex bodies of $\R^d$, with no restriction on their volume, location in the space and smoothness of the boundary.
\end{abstract}


\section{Introduction}

    Probabilistic properties of random polytopes have been studied extensively in the literature in the last fifty years. Consider a convex body $K$ in $\R^d$, and a set of $n$ i.i.d. random points uniformly distributed in $K$. The convex hull of these random points is a random polytope. Its number of vertices and its missing volume, i.e., the volume of its complement in $K$, have been first analyzed in the seminal work of Rényi and Sulanke \cite{RenyiSulanke63,RenyiSulanke64}. They derived the asymptotics of the expected missing volume in the case $d=2$, when $K$ is supposed to be either a polygon with a given number of vertices, or a convex set with smooth boundary. More recently, considerable efforts were devoted to understanding the behavior of the expected missing volume. Thus, several particular examples of $K$ were studied, including a $d$-dimensional simple polytope\footnote{A $d$-dimensional simple polytope is a convex polytope such that each of its vertices is adjacent to exactly $d$ edges.} \cite{AffentrangerWieacker91}, a $d$-dimensional polytope \cite{BaranyBuchta93} and a $d$-dimensional Euclidean ball \cite{BuchtaMuller84}. In \cite{Groemer74}, it is shown that the expected missing volume is maximal when $K$ is an ellipsoid, see also the references therein. B\'ar\'any and Larman \cite{BaranyLarman88} showed that if $K$ has volume one, then the expected missing volume has the same asymptotic behavior as the volume of the $(1/n)$-wet part of $K$, defined as the union of all caps of $K$ (a cap being the intersection of $K$ with a half space) of volume at most $1/n$. This reduces the initial probabilistic problem to computation of such a deterministic volume, which is a specific analytic problem that was extensively studied. When $K$ has a smooth boundary, a key point was the introduction of the affine surface area, see \cite{SchuttWerner90,Schutt93}, which leads to the result that the expected missing volume is of the order $n^{-2/(d+1)}$. When $K$ is a polytope, it is of the order $(\ln n)^{d-1}/n$ \cite{BaranyLarman88}. In addition, \cite{BaranyLarman88} proves that the expected missing volume, in dimension $d$, is minimal for simple polytopes, and maximal for ellipsoids. As a conclusion, the properties of the expected missing volume are now very well-understood. Much less is known about its higher moments and deviation probabilities. In particular, using a jackknife inequality for symmetric functions of $n$ random variables, Reitzner \cite{Reitzner03} proved that if $K$ is a $d$-dimensional smooth convex body, the variance of the missing volume is bounded from above by $n^{-(d+3)/(d+1)}$, and he conjectured that this is the actual order of magnitude for the variance. In addition, he proved that the second moment of the missing volume is exactly of the order $n^{-4/(d+1)}$, with explicit constants in terms of the affine surface area of $K$.
    Vu \cite{Vu05} obtained deviation inequalities for general convex bodies of volume one, involving quantities such as the volume of the wet part, and derived precise deviation inequalities in the cases when $K$ is a polytope, and when it has a smooth boundary. These inequalities involve constants which depend on $K$ in an unknown way. The main tools are martingale inequalities and, as a consequence, upper bounds on the moments of the missing volume are proved, again with implicit constants depending on $K$. Let $V_n$ stand for the missing volume of the convex hull of $n$ i.i.d. points uniformly distributed in a convex body $K$. If $K$ has a smooth boundary and is of volume one, Vu \cite{Vu05} showed the existence of positive constants $c$ and $\alpha$, which depend on $K$, such that for any $\lambda\in \left(0,(\alpha/4)n^{-\frac{(d-1)(d+3)}{(d+1)(3d+5)}}\right]$, the following holds:
    $$\PP\left[|V_n-\E[V_n]|\geq\sqrt{\alpha\lambda n^{-\frac{d+3}{d+1}}}\right]\leq 2\exp(-\lambda/4)+\exp\left(-cn^{\frac{d-1}{3d+5}}\right).$$
    This inequality allows one to derive upper bounds on the variance and on the $q$-th moment of the missing volume, respectively of orders $n^{-(d+3)/(d+1)}$ and $n^{-2q/(d+1)}$, for $q>0$, for a \textit{smooth convex body $K$ of volume one}, up to constant factors depending on $K$ in an unknown way. The aim of the present paper is to derive universal deviation inequalities and upper bounds on the moments of the missing volume, i.e., results with no restriction on the volume and boundary structure of $K$, and with constants which do not depend on $K$. The only assumptions on $K$ are compactness and convexity.

\section{Statement of the problem and notation}

    Let $d\geq 2$ be an integer. We denote by $|\cdot|$ the Lebesgue measure in $\R^d$, $\rho$ the Euclidean distance in $\R^d$, $B_d$ the unit Euclidean ball with center $0$, and $\beta_d$ its volume. \\

    If $G\subseteq\R^d$ and $\epsilon>0$, we denote by $G^\epsilon=\{x\in\R^d:\rho(x,G)\leq\epsilon\}$ the closed $\epsilon$-neighborhood of $G$. Here, $\displaystyle{\rho(x,G)=\inf_{y\in G}\rho(x,y)}$.\\

    When $G_1$ and $G_2$ are two subsets of $\R^d$, we denote by $G_1\triangle G_2$ their symmetric difference, and the Hausdorff distance between $G_1$ and $G_2$ is defined as:
    \begin{equation*}
        d_H(G_1,G_2)=\inf\{\epsilon>0:G_1\subseteq G_2^\epsilon, G_2\subseteq G_1^\epsilon\}.
    \end{equation*}
    For brevity, we call a convex body a compact and convex subset of $\R^d$ with positive Lebesgue measure. We denote by $\K$ the class of all convex bodies in $\R_d$, and by $\K^1$ the set of all convex bodies that are included in $B_d$. For a given $K\in\K$, consider a sample of $n$ i.i.d. random points $X_1,\ldots,X_n$, uniformly distributed in $K$. We denote by $\hat K_n$ the convex hull of $X_1,\ldots,X_n$. This is a random polytope whose missing volume is denoted by $V_n$, i.e., $V_n=|K\backslash\hat K_n|$. We denote respectively by $\PP_K$ and $\E_K$ the joint probability measure of $(X_1,\ldots,X_n)$ and the corresponding expectation operator. We are interested in deviation inequalities for $V_n$, i.e., in bounding from above the probability $$\PP_K[V_n>\epsilon|K|],$$
    for $\epsilon>0$. This yields, as a consequence, upper bounds for the moments $\E_K[V_n^q], q>0$. In order to obtain a deviation inequality, we use the metric entropy of the class $\K$. We first prove that it is sufficient to obtain a deviation inequality for $K\in\K^1$, by a scaling argument. The deviation inequality that we prove is uniform on the class $\K^1$, hence it is of much interest in a statistical framework. If one aims to recover $K$ from the observation of the sample points $X_1,\ldots,X_n$, using $\hat K_n$ as an estimator, the risk, measured in terms of the Nikodym distance (defined as the Lebesgue measure of the symmetric difference), can be bounded from above uniformly on $\K$, with no assumption on the volume, boundary structure and location in $\mathbb R^d$ of $K$.

\section{Deviation inequality for random polytopes}

        \begin{theorem}\label{DevIneq}
            There exist two positive constants $C_1$ and $C_2$, which depend on $d$ only, such that:
            \begin{equation}
                \label{Theorem1} \sup_{K\in\K}\PP_K\left[n\left(\frac{|K\backslash\hat K_n|}{|K|}-C_2n^{-2/(d+1)}\right)>x\right] \leq C_1e^{-x/\beta_d}, \forall x>0.
            \end{equation}
        \end{theorem}

        Theorem \ref{DevIneq} involves constants which depend at least exponentially on the dimension $d$. This seems to be the price for getting a uniform deviation inequality on $\K$. Note that the missing volume is normalized here by the volume of $K$. Theorem \ref{DevIneq} may be refined by normalizing the missing volume by another functional of $K$, which could be expressed in terms of the affine surface area of $K$, as in \cite{Schutt94} where only the first moment of $V_n$ is considered.

        Theorem \ref{DevIneq} allows one to derive upper bounds for all the moments of the missing volume. Indeed, applying Fubini's theorem leads to the following corollary.

        \begin{corollary}\label{Corollary1}
            For every positive number $q$, there exists some positive constant $A_q$, which depends on $d$ and $q$ only, such that
            \begin{equation}
                \label{Cor1} \E_K\left[|K\backslash\hat K_n|^q\right]\leq A_q|K|^q n^{-2q/(d+1)}, \forall K\in\K.
            \end{equation}
        \end{corollary}

        Note that no restriction is made on $K$ except for its compactness and convexity. In particular, its boundary may not be smooth, and $K$ may be located anywhere in the space, not necessarily in some given compact set. In this sense, the exponential deviation inequality \eqref{Theorem1} and the inequality on the moments \eqref{Cor1} are universal. Combining this corollary with Corollary 2 of \cite{Brunel13} yields the following result.

        \begin{corollary}\label{Corollary2}
            For every positive number $q$, there exist some positive constants $a_q$ and $A_q$, which depend on $d$ and $q$ only, such that
            $$a_qn^{-2q/(d+1)} \leq \sup_{K\in\K}\E_K\left[\left(\frac{|K\backslash\hat K_n|}{|K|}\right)^q\right]\leq A_q n^{-2q/(d+1)}.$$
        \end{corollary}

        In order to prove Theorem \ref{DevIneq}, we first state two lemmas.

        \begin{lemma}\label{ellipsoid}
            Let $K\in\K$. There exists an ellipsoid $E$ in $\R^d$ such that $K\subseteq E$ and $|E|\leq d^d|K|$.
        \end{lemma}
        Proof of Lemma \ref{ellipsoid} can be found in \cite{Leichtweiss59} and \cite{HugSchneider07}.
        The second lemma is based on the Steiner formula for convex bodies. It shows that on $\K^1$, the Nikodym distance is bounded from above by the Hausdorff distance, up to some positive constant.
        \begin{lemma}\label{Lemma2}
            There exists some positive constant $\alpha_1$ which depends on $d$ only, such that
            $$|G\triangle G'|\leq \alpha_1 d_H(G,G'), \forall G,G'\in\K^1.$$
        \end{lemma}

        \begin{proof}

            Let $G\in\K$. Steiner formula (see Section 4.1 in \cite{SchneiderBook}) states that there exist some positive numbers $L_1(G),\ldots,L_d(G)$, such that
            \begin{equation}
                \label{SteinerFormula}|G^\lambda\backslash G|=\sum_{j=1}^d L_j(G)\lambda^j, \lambda\geq0.
            \end{equation}
            Besides the $L_j(G), j=1,\ldots,d$ are increasing functions of $G$. In particular, if $G\in\K^1$, then $L_j(G)\leq L_j(B_d)$. \\
            Let $G,G'\in\K^1$, and let $\lambda=d_H(G,G')$. Since $G$ and $G'$ are included in the unit ball, $\lambda$ is not greater than its diameter, so $\lambda\leq 2$. By definition of the Hausdorff distance, $G\subseteq G'^{\lambda}$ and $G'\subseteq G^{\lambda}$. Hence,
            \begin{align*}
                |G\triangle G'| & = |G\backslash G'|+|G'\backslash G| \leq |G'^{\lambda}\backslash G'|+|G^{\lambda}\backslash G| \\
                & \leq 2\sum_{j=1}^d L_j(B_d)\lambda^j \leq \lambda\sum_{j=1}^d L_j(B_d)2^j.
            \end{align*}
            The Lemma is proved by setting $\alpha_1=\sum_{j=1}^d L_j(B_d)2^j$.

        Note that since $\delta\leq 1$, Steiner formula \eqref{SteinerFormula} implies, for $G\in\K^1$, that
        \begin{equation}
            \label{Steiner}|G^\delta\backslash G|\leq \alpha_2\delta,
        \end{equation}
        where $\alpha_2=\sum_{j=1}^d L_j(B_d)$.
        \end{proof}

    \paragraph{Proof of Theorem \ref{DevIneq}}

        This proof is inspired by Theorem 1 in \cite{KST}, which derives an upper bound on the risk of a convex hull type estimator of a convex function.
        Let $K\in\K$. Let $E$ be an ellipsoid which satisfies the properties of Lemma \ref{ellipsoid}, and $T$ an affine transform in $\R^d$ which maps $E$ to the unit ball $B_d$. Note that $\beta_d=|\det T| |E|$, so $T$ is invertible.
        Let us denote $K'=T(K)$ and $X'_i=T(X_i), i=1,\ldots,n$. Let $\hat K_n'$ be the convex hull of $X'_1,\ldots,X'_n$. By the definition of $T$, the following properties hold :
        \begin{enumerate}[(i)]
            \item $K'\in\K^{1}$, \label{prop1}
            \item $X'_1,\ldots,X'_n$ are i.i.d. uniformly distributed in $K'$, \label{prop2}
            \item $T(\hat K_n)=\hat K_n'$. \label{prop3}
        \end{enumerate}
        Furthermore, one has the following:
        \begin{equation}
            \label{IneqEllips}\frac{|K\backslash\hat K_n|}{|K|} = \frac{|K'\backslash\hat K_n'|}{|\det T||K|} = |K'\backslash\hat K_n'|\frac{|E|}{\beta_d|K|} \leq \frac{d^d}{\beta_d}|K'\backslash\hat K_n'|.
        \end{equation}
        Let $\delta=n^{-2/(d+1)}$. A $\delta$-net of $\K^1$, for the Hausdorff distance, is a collection of subsets of $\K^1$ such that for each $G\in\K^1$, there is $G^*$ in this collection of sets which satisfies $d_H(G,G^*)\leq\delta$. Bronshtein \cite{Bronshtein76} showed that there exists a finite $\delta$-net, of cardinality $N_\delta\leq C_1\delta^{-\frac{d-1}{2}}$ for some positive constant $C_1$. Let $\{G_1,\ldots,G_{N_\delta}\}$ be such a $\delta$-net. Let $j^*,\hat j\in\{1,\ldots,N_\delta\}$ be such that:
        \begin{equation*}
            d_H(K',G_{j^*})\leq\delta \mbox{ }\mbox{ }\mbox{ }\mbox{ }\mbox{ and}\mbox{ }\mbox{ }\mbox{ }\mbox{ } d_H(\hat K_n',G_{\hat j})\leq \delta.
        \end{equation*}

        Let $\varepsilon>0$. By \eqref{IneqEllips} and (\ref{prop2}),
        \begin{equation}
            \label{rescaling}\PP_K\left[\frac{|K\backslash\hat K_n|}{|K|}>\varepsilon\right]\leq\PP_{K'}\left[|K'\backslash\hat K_n'|>\frac{\beta_d}{d^d}\varepsilon\right].
        \end{equation}
        Let us recall that if $G, G'$ and $G''$ are three Borel subsets of $\R^d$, then the following triangle inequality holds:
        \begin{equation}
            \label{TI}|G\backslash G''|\leq |G\backslash G'|+|G'\backslash G''|.
        \end{equation}
        Thus, $|K'\backslash\hat K_n'|\leq |K'\backslash G_{j^*}|+|G_{j^*}\backslash G_{\hat j}|+|G_{\hat j}\backslash\hat K_n'|$ and, by the definition of $j^*$ and $\hat j$, and by Lemma \ref{Lemma2} and \eqref{rescaling},
        \begin{equation}
            \label{step1}\PP_K\left[\frac{|K\backslash\hat K_n|}{|K|}>\varepsilon\right] \leq \PP_{K'}\left[|G_{j^*}\backslash G_{\hat j}|>\frac{\beta_d}{d^d}\varepsilon-2\alpha_1\delta\right].
        \end{equation}
        Set $\varepsilon'=\frac{\beta_d}{d^d}\varepsilon-2\alpha_1\delta$. \eqref{step1} implies
        \begin{equation}
            \label{step2}\PP_K\left[\frac{|K\backslash\hat K_n|}{|K|}>\varepsilon\right] \leq \sum_{j=1,\ldots,N_\delta : |G_{j^*}\backslash G_{j}|>\varepsilon'}\PP_{K'}\left[\hat j=j\right].
        \end{equation}
        Let $j\in\{1,\ldots,N_\delta\}$ be fixed, such that $|G_{j^*}\backslash G_{j}|>\varepsilon'$. Recall that $\hat K_n'\subseteq G_{\hat j}^\delta$, and thus if $\hat j=j$, then $X_i'\in G_j^\delta, i=1,\ldots,n$. So,
        \begin{align*}
            \PP_{K'}\left[\hat j=j\right] & \leq \left(\PP_{K'}\left[X_1'\in G_j^\delta\right]\right)^n \\
            & \leq \left(1-\frac{|K'\backslash G_j^\delta|}{|K'|}\right)^n \\
            & \leq \left(1-\frac{1}{\beta_d}(|G_{j^*}\backslash G_j|-|G_{j^*}\backslash K'|-|G_j^\delta\backslash G_j|)\right)^n,
        \end{align*}
        using the triangle inequality \eqref{TI} and the fact that $|K'|\leq \beta_d$.
        Denote by $I_{\varepsilon'}=1$ if $\varepsilon'<\beta_d$, and $0$ otherwise. Continuing \eqref{step2}, and using \eqref{Steiner}, one gets:
        \begin{align*}
            \PP_K\left[\frac{|K\backslash\hat K_n|}{|K|}>\varepsilon\right] & \leq \sum_{j=1,\ldots,N_\delta : |G_{j^*}\backslash G_{j}|>\varepsilon'}\left(1-\frac{\varepsilon'}{\beta_d}+\frac{(\alpha_1+\alpha_2)\delta}{\beta_d}\right)^n \\
            & \leq N_\delta \left(1-\frac{\varepsilon'}{\beta_d}+\frac{(\alpha_1+\alpha_2)\delta}{\beta_d}\right)^nI_{\varepsilon'} \\
            & \leq C_1\exp\left(\delta^{-\frac{d-1}{2}}-\frac{\varepsilon n}{\beta_d}+\frac{(3\alpha_1+\alpha_2)\delta n}{\beta_d}\right) \\
            & \leq C_1\exp\left(\alpha_3\delta n-\frac{\varepsilon n}{\beta_d}\right),
        \end{align*}
        where $\alpha_3=1+\frac{3\alpha_1+\alpha_2}{\beta_d}$ is a positive constant which depends on $d$ only (recall that $\delta^{-\frac{d-1}{2}}=\delta n$).
        Finally, by choosing $\varepsilon=\alpha_3\beta_d\delta + x/n$, for any $x>0$, and by setting the constant $C_2=\alpha_3\beta_d$, one gets \eqref{Theorem1}.

\bibliographystyle{plain}
\bibliography{Biblio}

\end{document}